\documentclass{amsart}
\usepackage{amssymb}
\usepackage{bbold}
\usepackage{xcolor}
\usepackage[colorlinks=true,
linkcolor=black,
urlcolor=blue,
citecolor=black]{hyperref}
\usepackage[margin=1.0in]{geometry}
\usepackage{enumitem}
\usepackage{comment}
\usepackage[T1]{fontenc}
\allowdisplaybreaks
\theoremstyle{definition}
\newtheorem{theorem}{Theorem}[section]

\newtheorem{proposition}[theorem]{Proposition}
\newtheorem{lemma}[theorem]{Lemma}
\newtheorem{example}[theorem]{Example}
\newtheorem{rem}[theorem]{Remark}

\DeclareMathOperator{\Prep}{Prep}
\DeclareMathOperator{\lc}{lc}
\newcommand{\F}{\mathbb{F}}
\newcommand{\PP}{\mathcal{P}}
\newcommand{\hh}{\widehat h}
\begin{document}

\title[Simultaneously preperiodic points in characteristic $p$]{Simultaneously preperiodic points and colliding orbits for families of polynomials in characteristic $p$}
\author{Jungin Lee and GyeongHyeon Nam}
\date{}
\address{J. Lee -- Department of Mathematics, Ajou University, Suwon 16499, Republic of Korea \newline
\indent G. Nam -- Korea Institute for Advanced Study, 85 Hoegi-ro, Dongdaemun-gu, Seoul 02455, Republic of Korea}
\email{jileemath@ajou.ac.kr, ghnam@kias.re.kr}

\begin{abstract}
Let $k=\overline{\mathbb{F}}_p$ and let $f\in k[X]$ be a polynomial of degree at least $2$ such that $f-f(0)$ is non-additive. For $\alpha,\beta$ in an algebraically closed extension $L$ of $k$, we prove that there are infinitely many $\lambda\in L$ for which both points are preperiodic under $f(X)+\lambda$ if and only if $\alpha,\beta\in k$ or $f(\beta)-f(\alpha)\in E_f$. Here $E_f=\{c\in G_f:(f-f(0))^n(c)\in G_f\text{ for every }n\ge0\}$, where $G_f=\{c\in k : f(X+c)-f(X)\in k\}$. We also provide a necessary and sufficient criterion for the corresponding colliding orbits set to be infinite.
\end{abstract}
\maketitle

\vspace{-5mm}
\section{Introduction}\label{Sec1}

Let $L$ be an algebraically closed field of characteristic $p>0$, and let $k$ be the algebraic closure of the prime field $\F_p$ in $L$. We call elements of $k$ \emph{constants}. For a polynomial $\varphi$, write $\varphi^n$ for its $n$-th iterate with $\varphi^0(x)=x$. A point $x$ is \emph{preperiodic} under $\varphi$ if its forward orbit $\{\varphi^n(x):n\ge0\}$ is finite. For $f\in k[X]$ of degree $d\ge2$ and $\alpha,\beta\in L$, set
\begin{equation*}
f_\lambda(X):=f(X)+\lambda,\quad
\PP_f(\alpha):=\{\lambda\in L:\alpha\text{ is preperiodic under }f_\lambda\},
\end{equation*}
\begin{equation*}
\Prep(f;\alpha,\beta):=\PP_f(\alpha)\cap\PP_f(\beta).
\end{equation*}
We call a polynomial $f$ \emph{additive} if $f(X+Y)=f(X)+f(Y)$ as polynomials in $X$ and $Y$; in characteristic $p$, these are precisely the polynomials of the form $f(X)=\sum_i a_iX^{p^i}$.

For $f \in k[X]$ and an indeterminate $Y$, write
\begin{equation*}
D_Yf(X):=f(X+Y)-f(X)\in k[X,Y],
\end{equation*}
and let $D_cf(X)$ denote its specialization at $Y=c$ for $c \in L$.
Define
\begin{equation*}
\begin{aligned}
G_f&:=\{c\in k:D_cf(X)\in k\},\\
\Delta_f(c)&:=D_cf(X)=f(c)-f(0)\quad(c\in G_f),\\
E_f&:=\{c\in G_f:\Delta_f^n(c)\in G_f\text{ for every }n\ge0\}.
\end{aligned}
\end{equation*}
The last condition means that every iterate of the partially defined map $\Delta_f:G_f\to k$ remains in $G_f$. Every such orbit is finite since $c$ and the coefficients of $f$ are contained in a finite field. For a non-additive polynomial $f$ with $f(0)=0$, both $G_f$ and $E_f$ are finite $\F_p$-vector spaces; see Proposition \ref{prop35}.

Ghioca \cite[Theorem 1.1]{Ghi26} proved that, for $f(X)=X^d$ with $d$ not a power of $p$, the set $\Prep(f;\alpha,\beta)$ is infinite if and only if $\alpha,\beta\in k$ or $f(\alpha)=f(\beta)$. If $d$ is a power of $p$, i.e., $f$ is additive, then $\Prep(f;\alpha,\beta)$ is infinite if and only if $\beta-\alpha\in k$. For $f(X)=\sum_{i=1}^r c_iX^{d_i} \in k[X]$ with $c_1, \ldots, c_r \ne0$ and $1\le d_1<\cdots<d_r=d$, set $q_i:=p^{v_p(d_i)}$ and $A_i:=d_i-q_i$, where $v_p(m)$ denotes the exponent of $p$ in a positive integer $m$. In this setting, Ghioca and Hsia \cite[Theorem 1.1]{GH26} proved that $\Prep(f;\alpha,\beta)$ is infinite if and only if $\alpha,\beta\in k$ or $f(\alpha)=f(\beta)$ whenever $d\ge3$ and
\begin{equation*}
A_r>\max(1,A_1,\ldots,A_{r-1}).
\end{equation*}
Recently, the authors proved this criterion for binomials with $A_1>A_2$ \cite[Theorem 1.4]{LN25}. In the remaining case $A_1=A_2>0$, we obtained a weaker necessary condition allowing $f(\beta)-f(\alpha)$ to be a nonzero constant \cite[Theorem 1.5]{LN25}. In this paper, we extend these results to every non-additive polynomial, with the condition $f(\alpha)=f(\beta)$ replaced by $f(\beta)-f(\alpha)\in E_f$, as follows.

\begin{theorem}\label{thm11}
Let $f\in k[X]$ be a polynomial of degree at least $2$ such that $f-f(0)$ is non-additive. For $\alpha,\beta\in L$, the following conditions are equivalent:
\begin{enumerate}[label=(\roman*)]
\item $\Prep(f;\alpha,\beta)$ is infinite;
\item $\PP_f(\alpha)=\PP_f(\beta)$;
\item $\alpha,\beta\in k$ or $f(\beta)-f(\alpha)\in E_f$.
\end{enumerate}
\end{theorem}
Consequently, if $E_f=\{0\}$, then
\begin{equation}\label{eq1a}
|\Prep(f;\alpha,\beta)|=\infty
\quad\Leftrightarrow\quad
\alpha,\beta\in k\ \text{ or }\ f(\alpha)=f(\beta).
\end{equation}

The following example illustrates why the set $E_f$ is introduced.
\begin{example}\label{exmp12}
Let $p=2$ and $f(X)=X^6+X^5+X^3 \in k[X]$. For $c\in k$, a direct computation gives
\begin{equation*}
D_cf(X)=(c^2+c)X^4+(c^4+c)X^2+(c^4+c^2)X+f(c).
\end{equation*}
The coefficient of $X^4$ forces $c\in\F_2$ whenever $c\in G_f$. For these two values, $D_cf(X)=c$, so $G_f=E_f=\F_2$. Now let $L=\overline{k(t)}$, where $t$ is transcendental over $k$. Then $f(t+1)-f(t)=1$ so
\begin{equation*}
f_\lambda^n(t+1)-f_\lambda^n(t)=1
\end{equation*}
for every $n \ge 0$ and $\lambda \in L$. 
Hence $\PP_f(t)=\PP_f(t+1)$, and Lemma \ref{lem21} gives $|\Prep(f;t,t+1)|=\infty$, although $f(t)\ne f(t+1)$. Thus the equality criterion \eqref{eq1a} does not hold for arbitrary non-additive polynomials.
\end{example}

For a nonzero polynomial $H(Z)=\sum_{j=0}^m h_jZ^j$ with $h_m\ne0$, the symbols $\deg_ZH$ and $\lc_ZH$ denote its degree $m$ and the leading coefficient $h_m$, respectively. We omit the subscript when the variable is clear. We use $H'(Z)$ for the formal derivative of $H$ with respect to $Z$. We also write $f_T(X):=f(X)+T$ for an indeterminate parameter $T$.

Assume that $\Prep(f;\alpha,\beta)$ is infinite and that $\alpha,\beta$ are not both in $k$. The proof uses the local-height equality of Ghioca and Hsia, in the form recorded in \cite[Theorem 2.1]{LN25}. In Section \ref{Sec2}, we develop the tools needed for the proof of our main result. In Subsection \ref{Sub31}, we use the two fixed-point parameters $\alpha-f(\alpha)$ and $\beta-f(\beta)$ to prove that $f(\beta)-f(\alpha)\in E_f$, and then prove Theorem \ref{thm11}. Subsection \ref{Sub32} gives an effective procedure for computing $E_f$. Following the work of Asgarli and Ghioca \cite{AG26}, we study the colliding orbits problem for arbitrary non-additive polynomials in Section \ref{Sec4}; see Theorem \ref{thm41} for the main result.

\section{Height and degree estimates}\label{Sec2}

\subsection{Normalization and canonical heights}\label{Sub21}

For a polynomial $f\in k[X]$ of degree $d$ with leading coefficient $a_d$, choose $a\in k^*$ with $a^{d-1}=a_d$ and set
\begin{equation*}
g(X):=a(f(X/a)-f(0)),\quad \mu_a(X):=aX.
\end{equation*}
Then $g$ is monic, $g(0)=0$, and
\begin{equation}\label{eq2a}
\begin{aligned}
\mu_a\circ(f+\lambda)\circ\mu_a^{-1}&=g+a(\lambda+f(0)),\\
g(a\beta)-g(a\alpha)&=a(f(\beta)-f(\alpha)),\\
D_{ac}g(X)&=aD_cf(X/a).
\end{aligned}
\end{equation}
With the notation introduced above, the identities in \eqref{eq2a} give
\begin{equation*}
G_g=aG_f,\quad \Delta_g(ac)=a\Delta_f(c),\quad E_g=aE_f.
\end{equation*}
After applying \eqref{eq2a} and replacing $g$ with $f$, we may assume that $f$ is monic and $f(0)=0$ in the proof of Theorem \ref{thm11}. These normalizations remain in force until the end of Subsection \ref{Sub31}.

\begin{lemma}\label{lem21}
(\cite[Proposition 2.3 and the proof of Proposition 4.4]{GH26}) For every $\alpha\in L$, the set $\PP_f(\alpha)$ is infinite. If $\alpha\in k$, then $\PP_f(\alpha)=k$.
\end{lemma}

Suppose that $\Prep(f;\alpha,\beta)$ is infinite and that $\alpha,\beta$ are not both in $k$. For any $\lambda_0 \in \Prep(f;\alpha,\beta)$, both $\alpha$ and $\beta$ are algebraic over $k(\lambda_0)$. Hence
\begin{equation*}
\mathrm{trdeg}_k k(\alpha,\beta)\le1,
\end{equation*}
where $\mathrm{trdeg}_k$ denotes the transcendence degree over $k$. Every parameter making $\alpha$ preperiodic is algebraic over $k(\alpha)$. We may therefore choose $t\in L$ transcendental over $k$ such that both $\alpha$ and $\beta$ are algebraic over $k(t)$.
Write
\begin{equation*}
k(t)^{\mathrm{perf}}:=\bigcup_{j\ge0}k(t^{1/p^j})
\end{equation*}
for the perfect closure of $k(t)$. Following \cite[Section 2]{LN25}, let $K\subset L$ be any finite extension of $k(t)^{\mathrm{perf}}$ containing $\alpha$ and $\beta$, and let $\overline K$ be its algebraic closure in $L$. Let $\Omega_K$ be the set of pairwise inequivalent nontrivial places of $K$, and for each $v\in\Omega_K$, let $|\cdot|_v$ denote the corresponding normalized nonarchimedean absolute value. These absolute values are trivial on $k$. For $x\in K$, we have
\begin{equation}\label{eq2b}
|x|_v\le1\text{ for every }v\in\Omega_K\quad\Leftrightarrow\quad x\in k.
\end{equation}

For a place $v$, let $K_v$ be the completion of $K$ at $v$, and let $\mathbb C_v$ be the completion of an algebraic closure of $K_v$, and fix a $K$-embedding $\overline K\hookrightarrow\mathbb C_v$ for each $v$. We also write $|\cdot|_v$ for the extension of the absolute value to $\mathbb C_v$. For $r\ge0$, set $\log^+r:=\log\max(1,r)$. For $x,\lambda\in\mathbb C_v$, define the \emph{local canonical height} by
\begin{equation*}
\hh_{v,\lambda}(x):=\lim_{n\to\infty}d^{-n}\log^+|f_\lambda^n(x)|_v.
\end{equation*}

Normalize the absolute values $|\cdot|_v$ ($v \in \Omega_K$) so that the product formula holds and $\sum_{v\in\Omega_K}\log^+|t|_v=1$. Following \cite[Section 2.1]{LN25}, for $x\in\overline K^*$, define the absolute function-field height by
\begin{equation*}
h(x):=\frac{1}{[K(x):K]}\sum_{v\in\Omega_K}\sum_{y\in\mathrm{Gal}(\overline K/K)\cdot x}\log^+|y|_v,\quad h(0):=0.
\end{equation*}
In particular, $h(x)=\sum_{v\in\Omega_K}\log^+|x|_v$ for $x \in K$.

For $x,\lambda\in\overline K$, define the \emph{global canonical height} by
\begin{equation*}
\hh_\lambda(x):=\lim_{j\to\infty}d^{-j}h(f_\lambda^j(x)).
\end{equation*}
It satisfies $\hh_\lambda(f_\lambda^m(x))=d^m\hh_\lambda(x)$ for every $m\ge0$. Moreover, for every $x,\lambda\in K$, we have
\begin{equation*}
\hh_\lambda(x)=\sum_{v\in\Omega_K}\hh_{v,\lambda}(x).
\end{equation*}
The following theorem is the main arithmetic input for our proofs.

\begin{theorem}\label{thm22}
(\cite[Theorem 2.13]{GH26}, \cite[Theorem 2.1]{LN25}) Let $K$ be as above and let $\alpha,\beta\in K$. Assume that there exists an infinite sequence $(\lambda_n)_{n \ge 1}$ of distinct elements in $\overline{K}$ such that
\begin{equation*}
\lim_{n \to \infty}\hh_{\lambda_n}(\alpha)=\lim_{n \to \infty}\hh_{\lambda_n}(\beta)=0.
\end{equation*}
Then
\begin{equation}\label{eq2c}
\hh_{v,\lambda}(\alpha)=\hh_{v,\lambda}(\beta) \quad \text{for every } v \in \Omega_K \text{ and } \lambda\in \mathbb C_v.
\end{equation}
\end{theorem}

A preperiodic point has global canonical height zero, so the infinitude of $\Prep(f;\alpha,\beta)$ implies \eqref{eq2c}. In the local arguments below, $K$ may be any finite extension of $k(t)^{\mathrm{perf}}$ containing $\alpha$ and $\beta$. We emphasize that Theorem \ref{thm32} and Proposition \ref{prop34} require only \eqref{eq2c}, and not the infinitude of $\Prep(f;\alpha,\beta)$.

Since $f$ is monic with constant coefficients, \cite[Corollary 2.7]{GH26} gives $\hh_{v,0}(x)=\log^+|x|_v$. Consequently, we have
\begin{equation*}
\max(1,|\alpha|_v)=\max(1,|\beta|_v)=:C_v.
\end{equation*}
The set $S:=\{v\in\Omega_K:C_v>1\}$ is nonempty unless $\alpha,\beta \in k$. We will repeatedly use the escape estimate appearing in the proof of \cite[Lemma 2.6]{GH26}:
\begin{equation*}
|z|_v>R_\lambda:=\max(1,|\lambda|_v^{1/d})\quad\Rightarrow\quad |f_\lambda(z)|_v=|z|_v^d.
\end{equation*}
It follows that $\hh_{v,\lambda}(z)=0$ if and only if the orbit of $z$ under $f_\lambda$ is bounded, in which case the entire orbit is contained in the disk $|z|_v\le R_\lambda$.

\subsection{A coefficient bound}\label{Sub22}

Fix a place $v$, and recall that the absolute value $|\cdot|_v$ is trivial on $k$. We write $|\mathbb C_v^*|:=\{|z|_v:z\in\mathbb C_v^*\}$ for its value group.

\begin{lemma}\label{lem23}
Let $Q,H\in\mathbb C_v[T]$, where $Q$ is monic of degree $N\ge1$ and $H\ne0$ has degree $m \ge 0$. Let $R\in|\mathbb C_v^*|$ and $B>0$. If
\begin{equation*}
|Q(\lambda)|_v\le R\quad\Rightarrow\quad |H(\lambda)|_v\le B
\end{equation*}
for every $\lambda\in\mathbb C_v$, then $|\lc_T H|_v^N R^m\le B^N$.
\end{lemma}

\begin{proof}
The case $m=0$ is immediate, since $Q$ has a root. Suppose $m\ge1$, and write
\begin{equation*}
H(T)=a\prod_{i=1}^m(T-\rho_i).
\end{equation*}
Choose $\gamma\in\mathbb C_v^*$ with $|\gamma|_v=R$ and distinct $z_0,\ldots,z_m\in k$. The elements $y_j=\gamma z_j$ satisfy
\begin{equation*}
|y_j|_v\le R \quad \text{and} \quad |y_j-y_\ell|_v=R\quad(j\ne\ell).
\end{equation*}
For each $i$, at most one $y_j$ satisfies $|Q(\rho_i)-y_j|_v<R$. Thus there is an element $y\in\{y_0,\ldots,y_m\}$ such that $|Q(\rho_i)-y|_v\ge R$ for every $1 \le i \le m$. 
Write $Q(T)-y=\prod_{j=1}^N(T-\lambda_j)$. Since $|Q(\lambda_j)|_v=|y|_v\le R$, the hypothesis gives
\begin{equation*}
B^N\ge\prod_{j=1}^N|H(\lambda_j)|_v
=|a|_v^N\prod_{i=1}^m|Q(\rho_i)-y|_v
\ge |a|_v^N R^m. \qedhere
\end{equation*}
\end{proof}

\begin{lemma}\label{lem24}
Let $x\in\mathbb C_v$ and $C\ge\max(1,|x|_v)$. Then for every $n\ge1$ and $\lambda\in\mathbb C_v$,
\begin{equation}\label{eq2e}
|f_\lambda^n(x)|_v\le C
\quad\Leftrightarrow\quad
\hh_{v,\lambda}(x)\le d^{-n}\log C.
\end{equation}
\end{lemma}

\begin{proof}
If $|\lambda|_v>C^d$, then $|\lambda|_v>\max(1,|x|_v^d)$ so \cite[Lemma 2.6(iii)]{GH26} gives
\begin{equation*}
|f_\lambda^n(x)|_v=|\lambda|_v^{d^{n-1}}>C \quad \text{and} \quad 
\hh_{v,\lambda}(x)=\frac1d\log|\lambda|_v>\log C.
\end{equation*}
Thus both conditions in \eqref{eq2e} are false.

Assume that $|\lambda|_v\le C^d$, and set $w=f_\lambda^n(x)$. If $|w|_v>C$, then $|w|_v>R_\lambda$ so the escape estimate gives
\begin{equation*}
\hh_{v,\lambda}(x)=d^{-n}\log|w|_v>d^{-n}\log C.
\end{equation*}
If $|w|_v\le C$, then induction gives $|f_\lambda^j(w)|_v\le C^{d^j}$ for every $j\ge0$, so
\begin{equation*}
\hh_{v,\lambda}(x)=d^{-n}\hh_{v,\lambda}(w)\le d^{-n}\log C.
\qedhere
\end{equation*}
\end{proof}

\begin{proposition}\label{prop25}
Assume that \eqref{eq2c} holds, and let $P_n(T)=f_T^n(\beta)-f_T^n(\alpha)$ and $N_n=d^{n-1}$. For every $n\ge1$ with $P_n\ne0$, write $m_n:=\deg_TP_n$. Then for every $v \in \Omega_K$,
\begin{equation*}
|\lc_TP_n|_v\le C_v^{1-m_n/N_n}\le C_v.
\end{equation*}
\end{proposition}

\begin{proof}
Write $U_n(T)=f_T^n(\alpha)$, $V_n(T)=f_T^n(\beta)$, and $C=C_v$. Since $f$ is monic, both $U_n$ and $V_n$ are monic polynomials of degree $N_n$, so $m_n<N_n$. By \eqref{eq2c} and Lemma \ref{lem24}, we have
\begin{equation*}
|U_n(\lambda)|_v\le C\quad\Leftrightarrow\quad |V_n(\lambda)|_v\le C.
\end{equation*}
For every $\lambda \in \mathbb C_v$ satisfying these equivalent conditions, we have $|P_n(\lambda)|_v\le C$. Since $C\in |\mathbb C_v^*|$, Lemma \ref{lem23} applied to $Q=U_n$, $H=P_n$ and $R=B=C$ gives
\begin{equation*}
|\lc_TP_n|_v^{N_n}C^{m_n}\le C^{N_n}.
\end{equation*}
We conclude that $|\lc_TP_n|_v\le C^{1-m_n/N_n}\le C$.
\end{proof}

\subsection{Degree and leading-coefficient recurrences}\label{Sub23}

Since $f(X)$ is monic and $f(0)=0$, we have
\begin{equation*}
f(X)=\sum_{i=1}^r c_iX^{d_i}
\end{equation*}
for some $1\le d_1<\cdots<d_r=d$ and $c_1, \ldots, c_r \in k^*$ with $c_r=1$. Let $d_i=q_is_i$ with $q_i=p^{v_p(d_i)}$, $p \nmid s_i$, and let $A_i=d_i-q_i$. Let $e:=f(\beta)-f(\alpha)$, which is equal to $\epsilon_1$ in \cite{LN25}. 

\begin{lemma}\label{lem26}
Set
\begin{equation*}
U_n=f_T^n(\alpha),\quad V_n=f_T^n(\beta),\quad P_n=V_n-U_n,\quad N_n=d^{n-1}.
\end{equation*}
Assume that $|e|_v>1$. Then the following hold for every $n\ge1$.
\begin{enumerate}[label=(\alph*)]
\item $P_n\ne0$, and hence $m_n:=\deg_TP_n<N_n$ and $u_n:=\lc_TP_n$ are well defined. 

\item Define $I_n:=\{i:A_iN_n+q_im_n=m_{n+1}\}$ and $Q_n:=\max_{i\in I_n}q_i$. Then
\begin{equation}\label{eq2f}
\begin{aligned}
m_{n+1}&=\max_i(A_iN_n+q_im_n),\\
u_{n+1}&=\sum_{i\in I_n}c_is_i u_n^{q_i},\\
|u_{n+1}|_v&=|u_n|_v^{Q_n}.
\end{aligned}
\end{equation}
\end{enumerate}
\end{lemma}

\begin{proof}
Use induction on $n \ge 1$. Since $P_1=e$, the hypothesis $|e|_v>1$ gives $P_1=e\ne0$, $m_1=0$ and $|u_1|_v=|e|_v>1$.

Suppose that $P_n\ne0$ and $|u_n|_v>1$. Since the polynomials $U_n$ and $V_n$ are monic of degree $N_n$, we have $m_n<N_n$. For each $i$,
\begin{equation*}
V_n^{d_i}-U_n^{d_i}
=(U_n+P_n)^{d_i}-U_n^{d_i}
=\sum_{j=1}^{s_i}\binom{s_i}{j}
 U_n^{q_i(s_i-j)}P_n^{q_ij}.
\end{equation*}
Among the nonzero terms, the term with $j=1$ has the unique largest $T$-degree, since $m_n<N_n$. Its degree and leading coefficient are $A_iN_n+q_im_n$ and $s_iu_n^{q_i}$, respectively. Among the indices attaining $\max_i(A_iN_n+q_im_n)$, the values $q_i$ are distinct. Indeed, two such indices with the same $q_i$ would have the same $A_i$, and hence the same $d_i$. Since $|c_is_i|_v=1$ and $|u_n|_v>1$, the corresponding sum of leading coefficients has a unique term of largest absolute value and is therefore nonzero. Thus
\begin{equation*}
P_{n+1}=\sum_{i=1}^r c_i(V_n^{d_i}-U_n^{d_i})\ne0,
\quad m_{n+1}=\max_i(A_iN_n+q_im_n).
\end{equation*}
The indices attaining this maximum are precisely those in $I_n$, so $u_{n+1}=\sum_{i\in I_n}c_is_i u_n^{q_i}$. The term with exponent $Q_n$ has the unique largest absolute value, so $|u_{n+1}|_v=|u_n|_v^{Q_n}$. This completes the induction and proves \eqref{eq2f}.
\end{proof}

Assume that $|e|_v>1$, and define $\theta_n:=m_n/d^{n-1}$ for every $n\ge 1$. By Lemma \ref{lem26}, we have
\begin{equation*}
\theta_1=0,\quad \theta_{n+1}=\frac1d\max_i(A_i+q_i\theta_n),\quad
|u_n|_v=|e|_v^{Q_1\cdots Q_{n-1}}.
\end{equation*}

\begin{lemma}\label{lem27}
Assume that $f$ is non-additive, $|e|_v>1$, and the sequence $(|u_n|_v)_{n\ge1}$ is bounded. Then there is an index $1 \le \kappa \le r$ such that
\begin{equation}\label{eq2g}
q_\kappa=1,\quad M:=A_\kappa\ge1,\quad
\sigma:=\frac{M}{d-1}\in(0,1],\quad A_i\le(d-q_i)\sigma\quad(1\le i\le r).
\end{equation}
Moreover, $I_n=\{\kappa\}$ for all sufficiently large $n$, and $M>A_i$ for every $i\ne\kappa$.
\end{lemma}

\begin{proof}
Since $|e|_v>1$ and $|u_n|_v=|e|_v^{Q_1\cdots Q_{n-1}}$ is bounded, we have $Q_n=1$ for all sufficiently large $n$. For such $n$, every $i\in I_n$ satisfies $q_i=1$. Let $1 \le \kappa \le r$ be the unique index maximizing $A_i=d_i-1$ among all indices with $q_i=1$. Then we have
\begin{equation*}
I_n=\{\kappa\},\quad
\theta_{n+1}=\frac{M+\theta_n}{d},\quad
\lim_{n\to\infty}\theta_n=\sigma.
\end{equation*}
For every $i$, the inequality $A_i+q_i\theta_n\le M+\theta_n$ holds for all sufficiently large $n$. Taking $n \to \infty$ limits gives
\begin{equation*}
A_i+q_i\sigma\le M+\sigma=d\sigma.
\end{equation*}
Non-additivity gives $A_i>0$ for some $i$, hence $\sigma>0$. Also $M=d_\kappa-1\le d-1$, so $\sigma\le1$. Finally, if $q_i>1$, then
\begin{equation*}
A_i\le(d-q_i)\sigma=M-(q_i-1)\sigma<M.
\end{equation*}
If $q_i=1$ and $i\ne \kappa$, then $A_i<M$ by the choice of $\kappa$.
\end{proof}

\section{Simultaneous preperiodicity and translation groups}\label{Sec3}
In this section, we prove Theorem \ref{thm11}. First we prove that \eqref{eq2c} implies that $e=f(\beta)-f(\alpha) \in k$.

\begin{lemma}\label{lem31}
Assume that \eqref{eq2c} holds and $e\ne0$. Suppose that there is an index $\kappa$ satisfying \eqref{eq2g}. Then, for every $v\in S$, we have
\begin{equation*}
|e|_v\le C_v^{\sigma(1-M)}\le1.
\end{equation*}
\end{lemma}

\begin{proof}
Fix $v\in S$, write $C=C_v$ and $R=C^\sigma>1$. For $|x|_v=C$, write
\begin{equation*}
F_x(Z):=f(x+Z)-f(x)=\sum_{j=1}^d b_j(x)Z^j,\quad b_d(x)=1.
\end{equation*}
The $i$th monomial contributes a nonzero term to $b_j(x)$ only if $j\ge q_i$. Since $\sigma\le1$, \eqref{eq2g} gives
\begin{equation*}
d_i-j=A_i-(j-q_i)
\le(d-q_i)\sigma-(j-q_i)
\le(d-j)\sigma.
\end{equation*}
Since $q_\kappa=1$ and $A_\kappa>A_i$ for every $i\ne\kappa$, we obtain
\begin{equation}\label{eq3a}
|b_j(x)|_v\le R^{d-j}\quad(1\le j\le d),\quad
|b_1(x)|_v=|c_\kappa s_\kappa x^M|_v=C^M=R^{d-1}.
\end{equation}
Comparing the terms of $F_x$ using \eqref{eq3a} gives
\begin{equation*}
\begin{aligned}
0<|Z|_v<R&\quad\Rightarrow\quad |F_x(Z)|_v=R^{d-1}|Z|_v,\\
|Z|_v>R&\quad\Rightarrow\quad |F_x(Z)|_v=|Z|_v^d.
\end{aligned}
\end{equation*}
Thus every nonzero root of $F_x$ has absolute value $R$, and every bounded orbit under $F_x$ is contained in the disk $|Z|_v\le R$.

Let $x_1,\ldots,x_d$ be the roots of $f(X)=f(\alpha)$ counted with multiplicity. Since $|f(\alpha)|_v=C^d$, $|x_j|_v=C$ for each $j$. Moreover, \eqref{eq3a} gives $|f'(x_j)|_v=|b_1(x_j)|_v>0$, so they are distinct. For $i\ne j$, the difference $x_i-x_j$ is a nonzero root of $F_{x_j}$ so
\begin{equation}\label{eq3b}
|x_i-x_j|_v=R\quad(i\ne j).
\end{equation}
For each $j$, set $\lambda_j:=x_j-f(\alpha)=x_j-f(x_j)$ so that $f_{\lambda_j}^2(\alpha)=f_{\lambda_j}(\alpha)=x_j$. By \eqref{eq2c}, $\hh_{v,\lambda_j}(\beta)=0$ so the orbit of $\beta$ under $f_{\lambda_j}$ is bounded. By the relation
\begin{equation*}
f_{\lambda_j}^{n+1}(\beta)-x_j
=F_{x_j}(f_{\lambda_j}^n(\beta)-x_j),
\end{equation*}
the orbit of $f_{\lambda_j}(\beta)-x_j=e$ under $F_{x_j}$ is bounded. Hence
\begin{equation}\label{eq3c}
|D_ef(x_j)|_v = |F_{x_j}(e)|_v \le R \quad(1\le j\le d).
\end{equation}

For each $i$, the polynomial $(X+e)^{d_i}-X^{d_i}$ has degree $A_i$ and leading coefficient $s_i e^{q_i}$. Since $A_\kappa>A_i$ for every $i\ne\kappa$ and $q_\kappa=1$, the polynomial $H(X):=D_ef(X)$ satisfies
\begin{equation*}
\deg_XH=M<d,\quad \lc_XH=c_\kappa s_\kappa e.
\end{equation*}

Since $\deg_XH=M$, Lagrange interpolation at $x_1,\ldots,x_{M+1}$ gives
\begin{equation*}
\lc_XH =\sum_{j=1}^{M+1}\frac{H(x_j)}{\prod_{1 \le \ell \le M+1, \ell\ne j}(x_j-x_\ell)}.
\end{equation*}
By \eqref{eq3b}, each denominator has absolute value $R^M$. Hence \eqref{eq3c} gives
\begin{equation*}
|e|_v=|\lc_XH|_v \le R^{-M}\max_{1\le j\le M+1}|H(x_j)|_v
\le R^{1-M}=C^{\sigma(1-M)}\le1. \qedhere
\end{equation*}
\end{proof}

\begin{theorem}\label{thm32}
Let $f\in k[X]$ be non-additive. If \eqref{eq2c} holds, then $e \in k$.
\end{theorem}

\begin{proof}
Suppose that $e\notin k$. By \eqref{eq2b}, there exists $v\in\Omega_K$ such that $|e|_v>1$. Since $|e|_v\le C_v^d$, we have $v\in S$.
By Lemma \ref{lem26}, we have $P_n\ne0$ for every $n\ge1$. Proposition \ref{prop25} therefore gives $|u_n|_v\le C_v$ for every $n\ge1$ so the sequence $(|u_n|_v)_{n\ge1}$ is bounded. By Lemma \ref{lem27}, there is an index $1 \le \kappa \le r$ satisfying \eqref{eq2g}. Applying Lemma \ref{lem31} at $v$ gives $|e|_v\le1$, contradicting the choice of $v$.
\end{proof}

\subsection{Proof of Theorem \ref{thm11}}\label{Sub31}

For $c\in k$, write
\begin{equation} \label{eq3n1}
R_c(X):=f(X+D_cf(X))-f(X).
\end{equation}
In the following lemma, the zero polynomial is included among polynomials of degree at most one.

\begin{lemma}\label{lem33}
Let $c\in k$. If $\deg_XR_c\le1$ and $\deg_X R_{-c}\le1$, then $D_cf$ is constant.
\end{lemma}

\begin{proof}
If $\deg_XD_cf=m\ge2$, then $\deg_XR_c=dm>1$, a contradiction. Thus $\deg_XD_cf\le1$. Suppose that $D_cf(X)=aX+b$ with $a\ne0$. Then $c \ne 0$. Write $R_c(X)=rX+s$ and set $t=1+a$. By \eqref{eq3n1}, we have
\begin{equation*}
f(tX+b)-f(X)=rX+s.
\end{equation*}
Here $t\ne0$, since otherwise the left-hand side has degree $d\ge2$. A direct computation gives
\begin{equation*}
D_{tc}f(tX+b)=(f(X+c)+r(X+c)+s)-(f(X)+rX+s)=aX+b+rc.
\end{equation*}
Thus $D_{tc}f$ is affine with slope $a/t$. Since translation differences commute,
\begin{equation*}
atc=D_{tc}(D_cf)(X)=D_c(D_{tc}f)(X)=\frac{ac}{t}.
\end{equation*}
As $ac\ne0$, we obtain $(1+a)^2=t^2=1$.

Since $\deg_XR_{-c}\le1$ and
\begin{equation*}
D_{-c}f(X)=-D_cf(X-c)=-aX+ac-b,
\end{equation*}
the same argument with $c$ replaced by $-c$ gives $(1-a)^2=1$. Now we have $(1+a)^2=(1-a)^2=1$ and $a \ne 0$, which is impossible.
\end{proof}

\begin{proposition}\label{prop34}
Assume that \eqref{eq2c} holds and that $\alpha,\beta$ are not both in $k$. If
\begin{equation*}
e=f(\beta)-f(\alpha)\in k,
\end{equation*}
then $e\in E_f$.
\end{proposition}

\begin{proof}
Fix $v\in S$ and write $C=C_v>1$. Let $c_0=e$ and $c_{i+1}=f(c_i)$ for each $i\ge0$.

Suppose that $e\notin E_f$. Since $\Delta_f(c)=f(c)$ for $c\in G_f$, the definition of $E_f$ gives an index $j\ge0$ such that $c_i\in G_f$ for all $i<j$ and $c_j\notin G_f$. 
Thus for every $\lambda \in L$, we have
\begin{equation}\label{eq3d}
f_\lambda^{i+1}(\beta)-f_\lambda^{i+1}(\alpha)=c_i \quad(0\le i\le j).
\end{equation}
Let $n=j+1$, $c=c_j$ and set $\lambda_\alpha:=\alpha-f(\alpha)$ and $\lambda_\beta:=\beta-f(\beta)$. By \eqref{eq2c}, $\hh_{v,\lambda_\alpha}(\beta)=\hh_{v,\lambda_\beta}(\alpha)=0$. Lemma \ref{lem24}, together with $|\alpha|_v=|\beta|_v=C$, implies that
\begin{equation} \label{eq3n2}
|f_{\lambda_\alpha}^m(\beta)|_v\le C \quad \text{and} \quad 
|f_{\lambda_\beta}^m(\alpha)|_v\le C\quad \text{for every } m \ge 0.
\end{equation}

Applying \eqref{eq3d} to $\lambda=\lambda_\alpha$ and $i=j$, we obtain $f_{\lambda_\alpha}^n(\beta)=\alpha+c$, so
\begin{equation*}
\begin{aligned}
f_{\lambda_\alpha}^{n+1}(\beta) &=f(\alpha+c)+\alpha-f(\alpha)=\alpha+D_cf(\alpha), \\
f_{\lambda_\alpha}^{n+2}(\beta) &= f(\alpha+D_cf(\alpha))+\alpha-f(\alpha)=R_c(\alpha)+\alpha.
\end{aligned}
\end{equation*}
Hence $R_c(\alpha)=f_{\lambda_\alpha}^{n+2}(\beta)-\alpha$, and the same argument gives $R_{-c}(\beta)=f_{\lambda_\beta}^{n+2}(\alpha)-\beta$. Now \eqref{eq3n2} gives
\begin{equation*}
|R_c(\alpha)|_v, |R_{-c}(\beta)|_v \le C.
\end{equation*}
For every nonzero $H\in k[X]$ and $x\in\{\alpha,\beta\}$, we have $|H(x)|_v=C^{\deg_X H}$. Hence both $R_c$ and $R_{-c}$ have degree at most one. Lemma \ref{lem33} gives $D_cf\in k$, contradicting $c\notin G_f$. Thus $e\in E_f$.
\end{proof}

\begin{proof}[Proof of Theorem \ref{thm11}]
By \eqref{eq2a}, we may assume that $f$ is monic and $f(0)=0$. Suppose that (i) holds. If $\alpha, \beta$ are not both in $k$, then Theorem \ref{thm22} gives \eqref{eq2c}, Theorem \ref{thm32} gives $e\in k$, and Proposition \ref{prop34} gives $e\in E_f$. Thus (i) implies (iii).

Suppose that (iii) holds. If $\alpha,\beta \in k$, then Lemma \ref{lem21} gives $\PP_f(\alpha)=\PP_f(\beta)=k$. Otherwise, $e \in E_f$. Set $c_0=e$ and $c_{i+1}=\Delta_f(c_i)$ for each $i \ge 0$. For every $\lambda\in L$ and $n \ge 1$, we have
\begin{equation*}
f_\lambda^n(\beta)-f_\lambda^n(\alpha)=c_{n-1}.
\end{equation*}
Choose a finite subfield of $k$ containing $e$ and the coefficients of $f$. Since $c_{n+1}=f(c_n)$, all $c_n$ belong to this field. Consequently, the orbit of $\alpha$ under $f_\lambda$ is finite if and only if the orbit of $\beta$ under $f_\lambda$ is finite, which proves (ii). Finally, (ii) implies (i) by Lemma \ref{lem21}.
\end{proof}

\subsection{A procedure for computing \texorpdfstring{$E_f$}{Ef}}\label{Sub32}

For $f\in k[X]$, define subsets $G^{(n)}\subseteq G_f$ recursively by
\begin{equation*}
G^{(0)}:=G_f,\quad
G^{(n+1)}:=\{c\in G^{(n)}:\Delta_f(c)\in G^{(n)}\}
\quad(n\ge0).
\end{equation*}

\begin{proposition}\label{prop35}
Let $f(X)=\sum_{i=1}^d a_iX^i\in k[X]$ be a non-additive polynomial of degree $d\ge2$. For $1\le j<d$, define
\begin{equation*}
B_j(Y):=\sum_{i=j+1}^d a_i\binom{i}{j}Y^{i-j} \in k[Y].
\end{equation*}
Then the following hold.
\begin{enumerate}[label=(\alph*)]
\item $G_f$ is a finite $\F_p$-vector space, and the map $\Delta_f:G_f\to k$ is $\F_p$-linear. Moreover, we have
\begin{equation*}
G_f=\bigcap_{j=1}^{d-1}\{c\in k:B_j(c)=0\} \quad \text{and} \quad 
|G_f|\le d-1.
\end{equation*}

\item Each $G^{(n)}$ is an $\F_p$-vector subspace of $G_f$, and $E_f=G^{(\dim_{\F_p}G_f)}$. In particular, $E_f$ can be determined by solving finitely many polynomial equations over a finite field.
\end{enumerate}
\end{proposition}

\begin{proof}
\begin{enumerate}[label=(\alph*)]
\item For $1\le j<d$, the coefficient of $X^j$ in $D_Yf(X)$ is $B_j(Y)$, which gives the stated formula for $G_f$. At least one $B_j$ is nonzero, since otherwise $f(X+Y)-f(X)=f(Y)$ and $f$ would be additive. For such an index $j$, we have $\deg B_j\le d-j\le d-1$. Thus $G_f$ is finite and $|G_f|\le d-1$.
For $c,c'\in G_f$, we have
\begin{equation*}
D_{c+c'}f(X)=D_cf(X+c')+D_{c'}f(X)=\Delta_f(c)+\Delta_f(c').
\end{equation*}
Since $0\in G_f$, $G_f$ is an $\F_p$-vector space and the map $\Delta_f$ is $\F_p$-linear. 

\item Write $h:=\dim_{\F_p}G_f$. By the linearity of $\Delta_f$, the sets $G^{(n)}$ form a descending chain of $\F_p$-vector subspaces of $G_f$. Since $\Delta_f(c)=f(c)$ for $c\in G_f$, we also have
\begin{equation*}
G^{(n)}=\{c\in G_f:f^i(c)\in G_f\text{ for }0\le i\le n\}
\end{equation*}
for every $n \ge 0$. If $G^{(n+1)}=G^{(n)}$, then $\Delta_f(G^{(n)})\subseteq G^{(n)}$, so the formula above and the definition of $E_f$ give $G^{(n)}=E_f$. There are at most $h$ strict inclusions in the descending chain
\begin{equation*}
G_f=G^{(0)} \supseteq G^{(1)} \supseteq \cdots,
\end{equation*}
and therefore $G^{(h)}=E_f$. \qedhere
\end{enumerate}
\end{proof}

\section{The colliding orbits problem}\label{Sec4}

For $\alpha_1,\alpha_2,\beta\in L$, set
\begin{equation*}
C_f(\alpha_1,\alpha_2;\beta)
:=\{\lambda\in L:f_\lambda^m(\alpha_1)=f_\lambda^n(\alpha_2)=\beta
\text{ for some }m,n\ge1\}.
\end{equation*}
Asgarli and Ghioca \cite[Theorem 1.1]{AG26} gave a criterion for $C_f(\alpha_1,\alpha_2;\beta)$ to be infinite when $f(X)=X^d$, and the authors \cite[Theorem 1.8]{LN25} obtained a similar criterion for certain binomials.
In this section, we give a criterion for the infinitude of $C_f(\alpha_1,\alpha_2;\beta)$ when $f-f(0)$ is non-additive and $\alpha_1,\alpha_2,\beta$ are not all in $k$.

\subsection{Translation data and the collision criterion}\label{Sub41}

Assume that $f-f(0)$ is non-additive and $\deg f\ge2$. Since $G_f=G_{f-f(0)}$ and $E_f=E_{f-f(0)}$, Proposition \ref{prop35} shows that $E_f$ is a finite $\F_p$-vector space and $\Delta_f|_{E_f}$ is $\F_p$-linear. Write $\Delta:=\Delta_f|_{E_f}$, $\nu:=\dim_{\F_p}E_f$, and define
\begin{equation*}
N_f:=\ker\Delta^\nu,\quad W_f:=\Delta^\nu(E_f).
\end{equation*}
The kernels and images of the powers of $\Delta$ stabilize by exponent $\nu$, so
\begin{equation*}
\ker\Delta^{2\nu}=\ker\Delta^\nu,\quad \Delta^{2\nu}(E_f)=\Delta^\nu(E_f).
\end{equation*}
If $z=\Delta^\nu(y)\in N_f\cap W_f$, then $\Delta^{2\nu}(y)=\Delta^{\nu}(z)=0$, so $y\in\ker\Delta^{2\nu}=\ker\Delta^{\nu}$ and hence $z=0$. Moreover, the rank-nullity theorem gives $\dim_{\F_p}N_f+\dim_{\F_p}W_f=\nu$, so $E_f=N_f\oplus W_f$. The restriction $\Delta_W:=\Delta|_{W_f}$ is bijective. Let $\pi_f:E_f\to W_f$ be the projection with kernel $N_f$, and let $\ell_f$ be the order of $\Delta_W$.

We will repeatedly use the following polynomial identity:
\begin{equation}\label{eq4a}
f_T^n(X+z)=f_T^n(X)+\Delta^n(z) \quad(z\in E_f,\ n\ge0).
\end{equation}
The case $n=0$ is trivial. If the identity holds for $n\ge0$, then $\Delta^n(z)\in E_f\subseteq G_f$ so
\begin{equation*}
f_T^{n+1}(X+z)=f_T(f_T^n(X)+\Delta^n(z))=f_T(f_T^n(X))+\Delta(\Delta^{n}(z))=f_T^{n+1}(X)+\Delta^{n+1}(z).
\end{equation*}
Thus \eqref{eq4a} holds in $k[X,T]$ for every $n \ge 0$.

For $1\le r\le\ell_f$ and $\gamma\in W_f$, define the affine permutation and its orbit of zero by
\begin{equation*}
A_{r,\gamma}(z):=\Delta_W^r(z)+\gamma,\quad
\mathcal O_{r,\gamma}(0):=\{A_{r,\gamma}^j(0):j\ge0\}.
\end{equation*}
Since $\Delta_W$ is bijective, $A_{r,\gamma}$ is a permutation of $W_f$. For every $j\ge1$, we have
\begin{equation*}
A_{r,\gamma}^j(0)=\sum_{i=0}^{j-1}\Delta_W^{ri}(\gamma).
\end{equation*}
For $a,b\in L$, write $a\sim_f b$ if $f_T^s(a)=f_T^s(b)$ in $L[T]$ for some $s \ge 1$.

\begin{theorem}\label{thm41}
Let $f\in k[X]$ be a polynomial such that $f-f(0)$ is non-additive, and let $d:=\deg f\ge2$. Assume that $\alpha_1,\alpha_2,\beta\in L$ are not all in $k$. Then $C_f(\alpha_1,\alpha_2;\beta)$ is infinite if and only if at least one of the following conditions holds:
\begin{enumerate}[label=(\roman*)]
\item $\alpha_1\sim_f\alpha_2$;
\item For $b_i:=f(\alpha_i)-f(\beta)\in E_f$ and $w_i:=\Delta_W^{-1}(\pi_f(b_i))\in W_f$ for $i=1,2$, there exist an integer $1\le r\le\ell_f$ and $\gamma\in W_f$ such that $w_1,w_2\in\mathcal O_{r,\gamma}(0)$.
\end{enumerate}
If $\alpha_1,\alpha_2$ are not both in $k$, (i) is equivalent to the condition $f(\alpha_2)-f(\alpha_1)\in N_f$.
\end{theorem}

Under the normalization \eqref{eq2a}, the points $\alpha_1,\alpha_2,\beta$ are replaced by $a\alpha_1,a\alpha_2,a\beta$, and
\begin{equation*}
g_{a(T+f(0))}^s(a\alpha_1)-g_{a(T+f(0))}^s(a\alpha_2)=a(f_T^s(\alpha_1)-f_T^s(\alpha_2)).
\end{equation*}
Thus $\alpha_1\sim_f\alpha_2$ if and only if $a\alpha_1\sim_g a\alpha_2$, so (i) is preserved by the normalization.

For (ii), \eqref{eq2a} gives $g(a\alpha_i)-g(a\beta)=ab_i$, and we have
\begin{equation*}
E_g=aE_f,\quad W_g=aW_f,\quad \pi_g(ac)=a\pi_f(c)\quad(c\in E_f),\quad \ell_g=\ell_f.
\end{equation*}
Together with $\Delta_g(ac)=a\Delta_f(c)$, these identities show that $w_i$ is replaced by $aw_i$. Moreover,
\begin{equation*}
(\Delta_g|_{W_g})^r(az)+a\gamma=aA_{r,\gamma}(z)\quad(z,\gamma\in W_f),
\end{equation*}
so multiplication by $a$ maps the corresponding orbits of $0$ onto each other. Thus (ii) is also preserved by the normalization. We may therefore assume that $f$ is monic and $f(0)=0$ in Lemma \ref{lem42} and Subsections \ref{Sub43} and \ref{Sub44}.

\begin{lemma}\label{lem42}
If $a,b\in L$ are not both in $k$, then
\begin{equation*}
a\sim_f b\quad\Leftrightarrow\quad f(b)-f(a)\in N_f.
\end{equation*}
For arbitrary $a,b\in L$, the condition $f(b)-f(a)\in N_f$ implies that
\begin{equation}\label{eq4c}
f_T^n(a)=f_T^n(b)\quad (n\ge\nu+1).
\end{equation}
\end{lemma}

\begin{proof}
Write $e:=f(b)-f(a)$. If $e\in E_f$, then applying \eqref{eq4a} with $X=f_T(a)$ and $z=e$ gives
\begin{equation}\label{eq4d}
f_T^n(b)-f_T^n(a)=\Delta^{n-1}(e)\quad(n\ge1).
\end{equation}
Suppose that $a\sim_f b$. Then $\PP_f(a)=\PP_f(b)$, and this set is infinite by Lemma \ref{lem21}. Since $a,b$ are not both in $k$, Theorem \ref{thm11} gives $e\in E_f$. Equation \eqref{eq4d} implies $\Delta^j(e)=0$ for some $j\ge0$. Since the kernels stabilize by exponent $\nu$, we obtain $e\in N_f$. Conversely, if $e\in N_f$, then \eqref{eq4d} gives \eqref{eq4c}, and thus $a\sim_f b$. This implication holds for arbitrary $a,b\in L$.
\end{proof}

\subsection{Iterates along arithmetic progressions}\label{Sub42}

For any $x,y \in L$ and positive integers $a$ and $b$, we prove that there are infinitely many parameters $\lambda \in L$ such that $f_\lambda^n(x)=y$ for some positive integer $n\equiv a\pmod b$. Asgarli and Ghioca \cite[Theorem 2.4]{AG26} proved the case $f(X)=X^d$ without a restriction on $n$.

\begin{theorem}\label{thm43}
Let $f\in k[X]$ be a polynomial of degree $d\ge2$, and let $x,y\in L$. Then for any positive integers $a$ and $b$, the set
\begin{equation*}
\{\lambda\in L:f_\lambda^n(x)=y
\text{ for some }n\ge1\text{ with }n\equiv a\pmod b\}
\end{equation*}
is infinite.
\end{theorem}

\begin{proof}
For every $n \ge 1$, let $P_n(T):=f_T^n(x)-y$ and $Q_n(T):=f_T^n(y)-y$. Then
\begin{equation*}
\deg_T P_n=\deg_T Q_n=d^{n-1}.
\end{equation*}
For every $n,h \ge 1$, a direct computation gives
\begin{equation*}
P_{n+h}(T)-Q_h(T)=f_T^h(f_T^n(x))-f_T^h(y) \equiv0\pmod{P_n(T)}.
\end{equation*}
Thus $\gcd(P_n,P_{n+h})$ divides $Q_h$, which implies that 
\begin{equation} \label{eq4b}
\deg_T\gcd(P_n,P_{n+h})\le d^{h-1}.
\end{equation}

Suppose that the set of parameters in the statement is a finite set $S$ of size $s$. Since $L$ is algebraically closed and $P_a$ is nonconstant, there exists $\lambda \in L$ such that $P_a(\lambda)=0$ so $s\ge1$. Set $D:=d^{sb-1}$, and choose $n_0\ge1$ with $n_0\equiv a\pmod b$ and $d^{n_0-1}>sD$. Each of the $s+1$ polynomials
\begin{equation*}
P_{n_0}(T), P_{n_0+b}(T), \ldots, P_{n_0+sb}(T)
\end{equation*}
has degree greater than $sD$ and all its roots in $S$. Hence each has a root of multiplicity greater than $D$. By the pigeonhole principle, there exists $\lambda\in S$ which has multiplicity greater than $D$ in both $P_{n_0+ib}$ and $P_{n_0+jb}$ for some indices $0\le i<j\le s$. This contradicts \eqref{eq4b}, which gives
\begin{equation*}
\deg_T\gcd(P_{n_0+ib},P_{n_0+jb}) \le d^{(j-i)b-1}\le D. \qedhere
\end{equation*}
\end{proof}

We remark that Example \ref{exmp12} also provides a counterexample to \cite[Conjecture 1.7]{LN25} for general non-additive polynomials. Indeed, let $p,f,t,L$ be as in Example \ref{exmp12}. By Theorem \ref{thm43}, there are infinitely many $\lambda\in L$ for which $f_\lambda^\tau(t)=t+1$ for some $\tau\ge1$. Since $f_\lambda(X+1)=f_\lambda(X)+1$, we have
\begin{equation*}
f_\lambda^\tau(t+1)=t \quad \text{and} \quad f_\lambda^{2\tau}(t)=t.
\end{equation*}
Hence the set $C_f(t,t+1;t)$ is infinite, although $f(t) \ne f(t+1)$.

\subsection{Field reduction and small heights} \label{Sub43}
In this subsection, we give two lemmas which play important roles in proving Theorem \ref{thm41}.

\begin{lemma}\label{lem44}
Suppose that $\alpha_1,\alpha_2,\beta\in L$ are not all in $k$ and that $C_f(\alpha_1,\alpha_2;\beta)$ is infinite. Then either $\alpha_1\sim_f\alpha_2$, or both $\alpha_1\notin k$ and $\alpha_2,\beta\in\overline{k(\alpha_1)}$, where the algebraic closure is taken inside $L$.
\end{lemma}

\begin{proof}
For each $j\ge1$ and $i=1,2$, the equation $f_T^j(\alpha_i)=\beta$ has finitely many roots. Since $C_f(\alpha_1,\alpha_2;\beta)$ is infinite, for every integer $N\ge1$ there exist $\lambda\in L$ and integers $m,n\ge N$ such that 
\begin{equation*}
f_\lambda^m(\alpha_1)=f_\lambda^n(\alpha_2)=\beta.
\end{equation*}
Let $M$ be the algebraic closure of $k(\alpha_2)$ inside $L$. If $\alpha_1\notin M$, choose a nonarchimedean absolute value $| \cdot |$ on $L$ which is trivial on $M$ and satisfies $C:=|\alpha_1|>1$. Choose $\lambda \in L$ and positive integers $m$ and $n$ such that $f_\lambda^m(\alpha_1)=f_\lambda^n(\alpha_2)=\beta$ and $|\beta|<C^{d^{\min(m,n)}}$. This follows from the assumption that $C_f(\alpha_1,\alpha_2;\beta)$ is infinite.

If $|\lambda|<C^d$, then $|f_\lambda^m(\alpha_1)|=C^{d^m}$. If $|\lambda|\ge C^d$, then $|\alpha_2|\le1$ gives
\begin{equation*}
|f_\lambda^n(\alpha_2)|=|\lambda|^{d^{n-1}}\ge C^{d^n}.
\end{equation*}
Both cases contradict the bound $|\beta|<C^{d^{\min(m,n)}}$, so we conclude that $\alpha_1 \in M$. Interchanging $\alpha_1$ and $\alpha_2$ gives $M=\overline{k(\alpha_2)}=\overline{k(\alpha_1)}$.

If $\beta\notin M$, choose a nonarchimedean absolute value $| \cdot |$ on $L$ which is trivial on $M$ and satisfies $|\beta|>1$. The relation $f_\lambda^m(\alpha_1)=f_\lambda^n(\alpha_2)=\beta$ gives $|\lambda|>1$, so 
\begin{equation*}
|\beta|=|\lambda|^{d^{m-1}}=|\lambda|^{d^{n-1}}
\end{equation*}
and thus $m=n$. Moreover, $\lambda$ is transcendental over $M$, so $f_T^m(\alpha_1)=f_T^m(\alpha_2)$ in $L[T]$, i.e., $\alpha_1 \sim_f \alpha_2$. If $\beta \in M$, then $\alpha_2,\beta \in \overline{k(\alpha_1)}$. In this case, $\alpha_1\notin k$, since otherwise $\alpha_1,\alpha_2,\beta\in k$.
\end{proof}

\begin{lemma}\label{lem45}
Let $x,y,\lambda\in\overline{k(t)}$. If $f_\lambda^m(x)=y$ with $m\ge1$, then
\begin{equation*}
\hh_\lambda(x)\le\frac{h(x)+h(y)}{d^m}.
\end{equation*}
\end{lemma}

\begin{proof}
Choose a finite extension $K/k(t)^{\mathrm{perf}}$ containing $x,y,\lambda$, with the normalization of Subsection \ref{Sub21}. For each $v \in \Omega_K$, set $D_v:=\max(1,|x|_v,|y|_v)$. Since $|f_\lambda^m(x)|_v=|y|_v\le D_v$, Lemma \ref{lem24} gives
\begin{equation*}
\hh_{v,\lambda}(x)
\le d^{-m}\log D_v
\le d^{-m}(\log^+|x|_v+\log^+|y|_v).
\end{equation*}
Summing over $v\in\Omega_K$, we obtain
\begin{equation*}
\hh_\lambda(x)
=\sum_{v\in\Omega_K}\hh_{v,\lambda}(x)
\le d^{-m}\sum_{v\in\Omega_K}(\log^+|x|_v+\log^+|y|_v)
=\frac{h(x)+h(y)}{d^m}.\qedhere
\end{equation*}
\end{proof}

\subsection{Proof of the collision criterion} \label{Sub44}
 
\begin{proof}[Proof of Theorem \ref{thm41}]
Assume that $C_f(\alpha_1,\alpha_2;\beta)$ is infinite and that (i) does not hold. By Lemma \ref{lem44}, we have $\alpha_1\notin k$ and $\alpha_2,\beta\in\overline{k(\alpha_1)}$. Set $t=\alpha_1$, and choose a finite extension $K/k(t)^{\mathrm{perf}}$ containing $\alpha_1,\alpha_2,\beta$. Every $\lambda\in C_f(\alpha_1,\alpha_2;\beta)$ belongs to $\overline K = \overline{k(t)}$, since it is a root of $f_T^m(\alpha_1)-\beta$ for some $m\ge1$. For each $q\ge1$ and $i=1,2$, the equation $f_T^q(\alpha_i)=\beta$ has finitely many roots. Thus we may choose a sequence $(\lambda_j)_{j\ge1}$ of pairwise distinct elements of $C_f(\alpha_1,\alpha_2;\beta)$ and integers $m_j,n_j\ge j+\nu+1$ such that
\begin{equation*}
f_{\lambda_j}^{m_j}(\alpha_1)=f_{\lambda_j}^{n_j}(\alpha_2)=\beta
\end{equation*}
for each $j \ge 1$. Lemma \ref{lem45} gives
\begin{equation*}
\hh_{\lambda_j}(\alpha_1)
\le\frac{h(\alpha_1)+h(\beta)}{d^{m_j}}\to0, \quad 
\hh_{\lambda_j}(\alpha_2)
\le\frac{h(\alpha_2)+h(\beta)}{d^{n_j}}\to0.
\end{equation*}
Thus Theorem \ref{thm22} gives \eqref{eq2c} for $(\alpha_1, \alpha_2)$, so Theorem \ref{thm32} and Proposition \ref{prop34} imply
\begin{equation*}
e:=f(\alpha_2)-f(\alpha_1)\in E_f.
\end{equation*}
By \eqref{eq4a} and the assumption that (i) does not hold, we have
\begin{equation*}
f_T^q(\alpha_2)-f_T^q(\alpha_1)=\Delta^{q-1}(e)\ne0
\end{equation*}
for every $q \ge 1$. In particular, $m_j\ne n_j$ for every $j$.

By \eqref{eq4d} and the equation $\hh_\lambda(f_\lambda^m(x))=d^m\hh_\lambda(x)$, we have
\begin{equation*}
d^{-m_j}\hh_{\lambda_j}(\beta)
=\hh_{\lambda_j}(\alpha_1)
=\hh_{\lambda_j}(\alpha_2)
=d^{-n_j}\hh_{\lambda_j}(\beta)
\end{equation*}
for every $j \ge 1$. Since $m_j\ne n_j$, all three global canonical heights vanish for every $j\ge1$. 
Thus Theorem \ref{thm22} gives \eqref{eq2c} for $(\alpha_1, \beta)$, so Theorem \ref{thm32} and Proposition \ref{prop34} imply
\begin{equation*}
b_1=f(\alpha_1)-f(\beta)\in E_f,\quad b_2=b_1+e\in E_f.
\end{equation*}
Recall that $w_i:=\Delta_W^{-1}(\pi_f(b_i)) \in W_f$. Since $\Delta^\nu(b_i-\pi_f(b_i))=0$, \eqref{eq4a} and \eqref{eq4d} give
\begin{equation}\label{eq4f}
f_\lambda^q(\alpha_i)
=f_\lambda^q(\beta)+\Delta_W^q(w_i)
=f_\lambda^q(\beta+w_i)
\end{equation}
for every $\lambda \in L$, $q \ge \nu+1$ and $i=1,2$.

Fix $j\ge1$, and set $F=f_{\lambda_j}$, $m=m_j$, $n=n_j$ and $Y_\beta:=\beta+W_f$. Since $m,n\ge\nu+1$, \eqref{eq4f} gives
\begin{equation*}
F^m(\beta+w_1)=F^n(\beta+w_2)=\beta.
\end{equation*}
Since $w_1,w_2\in W_f$ and $\Delta_W$ is bijective,
\eqref{eq4a} gives
\begin{equation*}
F^m(Y_\beta)
=F^m(\beta+w_1)+\Delta_W^m(W_f)
=\beta+W_f=Y_\beta.
\end{equation*}
Similarly, $F^n(Y_\beta)=Y_\beta$. By the Euclidean algorithm, $F^\tau(Y_\beta)=Y_\beta$ for $\tau:=\gcd(m,n)$, so $\gamma:=F^\tau(\beta)-\beta\in W_f$. Choose $1\le r\le\ell_f$ such that $\tau\equiv r\pmod{\ell_f}$. 

By \eqref{eq4a},
\begin{equation*}
F^\tau(\beta+z)=\beta+\gamma+\Delta_W^r(z)
=\beta+A_{r,\gamma}(z)\quad(z\in W_f).
\end{equation*}
By induction on $s$, we obtain
\begin{equation}\label{eq4g}
F^{\tau s}(\beta+z)=\beta+A_{r,\gamma}^s(z)
\quad(z\in W_f,\ s\ge0),
\end{equation}
which gives $A_{r,\gamma}^{m/\tau}(w_1)=0$ and $A_{r,\gamma}^{n/\tau}(w_2)=0$. Since $A_{r,\gamma}$ is a permutation of $W_f$, we have $w_1,w_2\in\mathcal O_{r,\gamma}(0)$. This proves (ii).

Now we prove the converse. Suppose that (i) holds, and choose $s\ge1$ such that $f_T^s(\alpha_1)=f_T^s(\alpha_2)$. By Theorem \ref{thm43} with $a=b=s$, there are infinitely many $\lambda \in L$ such that $f_\lambda^n(\alpha_1)=\beta$ for some positive multiple $n$ of $s$. Since $n\ge s$, we also have $f_\lambda^n(\alpha_2)=\beta$. Hence $C_f(\alpha_1,\alpha_2;\beta)$ is infinite.

Suppose now that (ii) holds, and choose $r$ and $\gamma$ as in the statement. Since $w_i\in\mathcal O_{r,\gamma}(0)$ and $A_{r,\gamma}$ is a permutation of $W_f$, we may choose $s_1,s_2\ge\nu+1$ such that $A_{r,\gamma}^{s_i}(w_i)=0$ for $i=1,2$. By Theorem \ref{thm43}, there are infinitely many $\lambda\in L$ such that
\begin{equation*}
f_\lambda^\tau(\beta)=\beta+\gamma\quad\text{for some }\tau\ge1\text{ with }\tau\equiv r\pmod{\ell_f}.
\end{equation*}
For each such $\lambda$, $\tau s_i\ge\nu+1$ so the same induction as in \eqref{eq4g}, together with \eqref{eq4f}, gives
\begin{equation*}
f_\lambda^{\tau s_i}(\alpha_i)
=f_\lambda^{\tau s_i}(\beta+w_i)
=\beta+A_{r,\gamma}^{s_i}(w_i)
=\beta\quad(i=1,2).
\end{equation*}
Hence $C_f(\alpha_1,\alpha_2;\beta)$ is infinite.

Finally, if $\alpha_1,\alpha_2$ are not both in $k$, the equivalence of (i) with the condition $f(\alpha_2)-f(\alpha_1)\in N_f$ follows from Lemma \ref{lem42}.
\end{proof}

\begin{rem}\label{rmk46}
For any non-additive binomial $f(X)=c_1X^{d_1}+c_2X^{d_2}\in k[X]$ with $0<d_1<d_2$, write $d_i=q_is_i$ and $A_i=d_i-q_i$. Then $G_f=\{0\}$. Indeed, let $0 \ne c \in k$. If $A_1\ne A_2$, then $\deg_X D_cf=\max(A_1,A_2)>0$, so $D_cf\notin k$. If $A_1=A_2>0$, then $q_1<q_2$, and the coefficient of $X^{q_1}$ in $D_cf$ is $c_1s_1c^{A_1}\ne0$. Hence $D_cf\notin k$ in either case.

Consequently, $E_f=N_f=W_f=\{0\}$. Thus Theorem \ref{thm11} recovers \cite[Theorem 1.4]{LN25} and strengthens \cite[Theorem 1.5]{LN25}. For the collision result, if $\alpha_1,\alpha_2$ are not both in $k$, Lemma \ref{lem42} gives $\alpha_1\sim_f\alpha_2$ if and only if $f(\alpha_1)=f(\alpha_2)$. The remaining case $\alpha_1,\alpha_2\in k$ and $\beta\notin k$ in \cite[Theorem 1.8]{LN25} is handled by \cite[Lemma 5.3]{LN25}. Hence Theorem \ref{thm41} recovers \cite[Theorem 1.8]{LN25}.
\end{rem}

\bigskip
\section*{Acknowledgments}
Jungin Lee was supported by the National Research Foundation of Korea (NRF) grant funded by the Korea government (MSIT) (No. RS-2024-00334558 and No. RS-2025-02262988).
GyeongHyeon Nam was supported by a KIAS Individual Grant (HP111501) at Korea Institute for Advanced Study.

\bigskip
\section*{Statement on AI use}

Building on research questions posed by the authors, OpenAI's ChatGPT 6 Pro generated the initial proofs of Theorems \ref{thm11} and \ref{thm41} through an iterative dialogue with the authors.
The authors formulated the project, evaluated and revised the model's suggestions, independently verified all mathematical arguments, and wrote the final manuscript. The authors take full responsibility for the correctness of the results.


\end{document}